\documentclass[a4paper,12pt]{amsart}

\usepackage{amsmath}
\usepackage{amssymb}
\usepackage{mathrsfs}
\usepackage{enumerate}
\usepackage{ifthen}
\usepackage{graphicx}
\usepackage[T1]{fontenc} 

\nonstopmode \numberwithin{equation}{section}
\newtheorem{thm}{Theorem}[section]
\newtheorem{cor}{Corollary}[section]
\newtheorem{lem}{Lemma}[section]
\newtheorem{prop}[equation]{Proposition}

\newtheorem{conj}{Conjecture}

\theoremstyle{definition}

\newtheorem{prob}{Problem}[section]
\newtheorem{rem}{Remark}[section]

\newcounter {own}
\def\theown {\thesection       .\arabic{own}}

\newenvironment{pf}[1][]{%
 \vskip 3mm
 \noindent
 \ifthenelse{\equal{#1}{}}%
  {{\slshape Proof. }}%
  {{\slshape #1.} }%
 }%
{\qed\bigskip}

\newcounter{alphabet}
\newcounter{tmp}

\def\be{\begin{equation}}
\def\ee{\end{equation}}

\newcommand{\bee}{\begin{enumerate}}
\newcommand{\eee}{\end{enumerate}}

\newcommand{\blem}{\begin{lem}}
\newcommand{\elem}{\end{lem}}
\newcommand{\bthm}{\begin{thm}}
\newcommand{\ethm}{\end{thm}}
\newcommand{\bcor}{\begin{cor}}
\newcommand{\ecor}{\end{cor}}
\newcommand{\beg}{\begin{examp}}
\newcommand{\eeg}{\end{examp}}
\newcommand{\begs}{\begin{examples}}
\newcommand{\eegs}{\end{examples}}
\newcommand{\bdefe}{\begin{defin}}
\newcommand{\edefe}{\end{defin}}
\newcommand{\bprob}{\begin{prob}}
\newcommand{\eprob}{\end{prob}}
\newcommand{\bei}{\begin{itemize}}
\newcommand{\eei}{\end{itemize}}

\newcommand{\bcon}{\begin{conj}}
\newcommand{\econ}{\end{conj}}
\newcommand{\bcons}{\begin{conjs}}
\newcommand{\econs}{\end{conjs}}
\newcommand{\bprop}{\begin{prop}}
\newcommand{\eprop}{\end{prop}}
\newcommand{\br}{\begin{rem}}
\newcommand{\er}{\end{rem}}
\newcommand{\brs}{\begin{rems}}
\newcommand{\ers}{\end{rems}}
\newcommand{\bo}{\begin{obser}}
\newcommand{\eo}{\end{obser}}
\newcommand{\bos}{\begin{obsers}}
\newcommand{\eos}{\end{obsers}}
\newcommand{\bpf}{\begin{pf}}
\newcommand{\epf}{\end{pf}}
\newcommand{\ba}{\begin{array}}
\newcommand{\ea}{\end{array}}
\newcommand{\beq}{\begin{eqnarray}}
\newcommand{\beqq}{\begin{eqnarray*}}
\newcommand{\eeq}{\end{eqnarray}}
\newcommand{\eeqq}{\end{eqnarray*}}

\begin{document}

\title{On logarithmic coefficients of some close-to-convex functions}

\author{Md Firoz Ali}
\address{Md Firoz Ali,
Department of Mathematics,
Indian Institute of Technology Kharagpur,
Kharagpur-721 302, West Bengal, India.}
\email{ali.firoz89@gmail.com}

\author{A. Vasudevarao}
\address{A. Vasudevarao,
Department of Mathematics,
Indian Institute of Technology Kharagpur,
Kharagpur-721 302, West Bengal, India.}
\email{alluvasu@maths.iitkgp.ernet.in}



\begin{abstract}
The logarithmic coefficients $\gamma_n$ of an analytic and univalent function $f$ in the unit disk $\mathbb{D}=\{z\in\mathbb{C}:|z|<1\}$ with the normalization $f(0)=0=f'(0)-1$ is defined by $\log \frac{f(z)}{z}= 2\sum_{n=1}^{\infty} \gamma_n z^n$. Recently, D.K. Thomas [On the logarithmic coefficients of close to convex functions, {\it Proc. Amer. Math. Soc.} {\bf 144} (2016), 1681--1687] proved that $|\gamma_3|\le \frac{7}{12}$ for functions in a subclass of close-to-convex functions (with argument $0$) and claimed that the estimate is sharp by providing a form of a extremal function. In the present paper, we pointed out that such extremal functions do not exist and the estimate is not sharp by providing a much more improved bound for the whole class of close-to-convex functions (with argument $0$). We also determine a sharp upper bound of $|\gamma_3|$ for close-to-convex functions (with argument $0$) with respect to the Koebe function.

\end{abstract}

\thanks{}

\maketitle
\pagestyle{myheadings}
\markboth{Md Firoz Ali and A. Vasudevarao }{On logarithmic coefficients of some close-to-convex functions}

\section{Introduction}

Let $\mathcal{A}$ denote the class of analytic functions $f$ in the unit disk $\mathbb{D}=\{z\in\mathbb{C}:|z|<1\}$ normalized by $f(0)=0=f'(0)-1$. If $f\in\mathcal{A}$ then $f(z)$ has the following representation
\begin{equation}\label{p6-001}
f(z)= z+\sum_{n=2}^{\infty}a_n(f) z^n.
\end{equation}
We will simply write $a_n:=a_n(f)$ when there is no confusion. Let $\mathcal{S}$ denote the class of all univalent (i.e. one-to-one) functions in $\mathcal{A}$. A function $f\in\mathcal{A}$ is called starlike (convex respectively) if $f(\mathbb{D})$ is starlike with respect to the origin (convex respectively). Let $\mathcal{S}^*$ and $\mathcal{C}$ denote the class of starlike and convex functions in $\mathcal{S}$ respectively. It is well-known that a function $f\in\mathcal{A}$ is in $\mathcal{S}^*$ if and only if ${\rm Re\,}\left(zf'(z)/f(z)\right)>0$ for $z\in\mathbb{D}$.
Similarly, a function $f\in\mathcal{A}$ is in $\mathcal{C}$ if and only if ${\rm Re\,}\left(1+(zf''(z)/f'(z))\right)>0$ for $z\in\mathbb{D}$.
From the above it is easy to see that $f\in\mathcal{C}$ if and only if $zf'\in\mathcal{S}^*$. Given $\alpha\in(-\pi/2,\pi/2)$ and $g\in\mathcal{S}^*$, a function $f\in\mathcal{A}$ is said to be close-to-convex with argument $\alpha$ and with respect to $g$ if
\begin{equation}\label{p6-003}
{\rm Re\,} \left(e^{i\alpha}\frac{zf'(z)}{g(z)}\right)>0 \quad z\in\mathbb{D}.
\end{equation}
Let $\mathcal{K}_{\alpha}(g)$ denote the class of all such functions. Let
$$
\mathcal{K}(g):= \bigcup_{\alpha\in(-\pi/2,\pi/2)} \mathcal{K}_{\alpha}(g) \quad\mbox{ and }\quad \mathcal{K}_{\alpha}:= \bigcup_{g\in\mathcal{S}^*} \mathcal{K}_{\alpha}(g)
$$
be the classes of functions called close-to-convex functions with respect to $g$ and close-to-convex functions with
argument $\alpha$, respectively. The class
$$
\mathcal{K}:= \bigcup_{\alpha\in(-\pi/2,\pi/2)} \mathcal{K}_{\alpha}= \bigcup_{g\in\mathcal{S}^*} \mathcal{K}(g)
$$
is the class of all close-to-convex functions. It is well-known that every close-to-convex function is univalent in $\mathbb{D}$ (see \cite{Duren-book}). Geometrically, $f\in\mathcal{K}$ means that the complement of the image-domain $f(\mathbb{D})$ is the union of non-intersecting half-lines.

The logarithmic coefficients of $f\in\mathcal{S}$ are defined by
\begin{equation}\label{p6-005}
\log \frac{f(z)}{z}= 2\sum_{n=1}^{\infty} \gamma_n z^n
\end{equation}
where $\gamma_n$ are known as the logarithmic coefficients. The logarithmic coefficients $\gamma_n$ play a central role in the theory of univalent functions. Very few exact upper bounds for $\gamma_n$ seem have been established. The significance of this problem in the context of Bieberbach conjecture was pointed out by Milin in his conjecture. Milin conjectured that for $f\in\mathcal{S}$ and $n\ge 2$,
$$
\sum_{m=1}^{n}\sum_{k=1}^{m} \left(k|\gamma_k|^2-\frac{1}{k}\right)\le 0,
$$
which led De Branges, by proving this conjecture, to the proof of the Bieberbach conjecture \cite{Branges-1985}. More attention has been given to the results of an average sense (see \cite{Duren-book,Duren-Leung-1979}) than the exact upper bounds for $|\gamma_n|$. For the Koebe function $k(z)=z/(1-z)^2$, the logarithmic coefficients are $\gamma_n=1/n$. Since the Koebe function $k(z)$ plays the role of extremal function for most of the extremal problems in the class $\mathcal{S}$, it is expected that $|\gamma_n|\le \frac{1}{n}$ holds for functions in $\mathcal{S}$. But this is not true in general, even in order of magnitude \cite[Theorem 8.4]{Duren-book}. Indeed, there exists a bounded function $f$ in the class $\mathcal{S}$ with logarithmic coefficients $\gamma_n\ne O(n^{-0.83})$ (see \cite[Theorem 8.4]{Duren-book}).

By differentiating (\ref{p6-005}) and equating coefficients we obtain
\begin{equation}\label{p6-010}
\gamma_1=\frac{1}{2} a_2
\end{equation}
\begin{equation}\label{p6-015}
\gamma_2=\frac{1}{2}(a_3-\frac{1}{2}a_2^2)
\end{equation}
\begin{equation}\label{p6-020}
\gamma_3=\frac{1}{2}(a_4-a_2a_3+\frac{1}{3}a_2^3).
\end{equation}
If $f\in\mathcal{S}$ then $|\gamma_1|\le 1$ follows at once from (\ref{p6-010}). Using Fekete-Szeg\"{o} inequality \cite[Theorem 3.8]{Duren-book} in (\ref{p6-015}), we can obtain the sharp estimate
$$
|\gamma_2|\le \frac{1}{2}(1+2e^{-2})=0.635\ldots.
$$
For $n\ge 3$, the problem seems much harder, and no significant upper bound for $|\gamma_n|$ when $f\in\mathcal{S}$ appear to be known.

If $f\in\mathcal{S}^*$ then it is not very difficult to prove that $|\gamma_n|\le \frac{1}{n}$ for $n\ge 1$ and equality holds for the Koebe function $k(z)=z/(1-z)^2$.  The inequality $|\gamma_n|\le \frac{1}{n}$ for $n\ge 2$ extends to the class $\mathcal{K}$ was claimed in a paper of Elhosh \cite{Elhosh-1996}. However, Girela \cite{Girela-2000} pointed out some error in the proof of Elhosh \cite{Elhosh-1996} and, hence, the result is not substantiated. Indeed, Girela proved that for each $n\ge 2$, there exists a function $f\in\mathcal{K}$ such that $|\gamma_n|> \frac{1}{n}$. In the same paper it has been shown that $|\gamma_n|\le \frac{3}{2n}$ holds for $n\ge 1$ whenever $f$ belongs to the set of extreme points of the closed convex hull of the class $\mathcal{K}$. Recently, Thomas \cite{Thomas-2016} proved that $|\gamma_3|\le \frac{7}{12}$ for functions in $\mathcal{K}_0$ (close-to-convex functions with argument $0$) with the additional assumption that the second coefficient of the corresponding starlike function $g$ is real. Thomas claimed that this estimate is sharp and has given a form of the extremal function. But after rigorous reading of the paper \cite{Thomas-2016}, we observed that such functions do not belong to the class $\mathcal{K}_0$ (more details will be given in Section \ref{Main Results}).

By  fixing a  starlike function $g$ in the class $\mathcal{S}^*$, the inequality (\ref{p6-003}) assertions a specific  subclass of close-to-convex functions. One of such important subclass is the class of close-to-convex functions with respect to the Koebe function $k(z)=z/(1-z)^2$. In this case, the inequality (\ref{p6-003}) becomes
\begin{equation}\label{p6-021}
{\rm Re\,} \left(e^{i\alpha}(1-z)^2f'(z)\right)>0,\quad z\in\mathbb{D}
\end{equation}
and defines the subclass $\mathcal{K}_{\alpha}(k)$. Several authors have been extensively studied the class of functions $f\in\mathcal{S}$ that satisfies the  condition (\ref{p6-021})  (see \cite{Elin-Khavinson-Reich-Shoikhet-2010,Hengartner-Schober-1970,Kowalczyk-Lecko-2014,Marjono-Thomas-2016}). Geometrically (\ref{p6-021}) says that the function $h:=e^{i\delta}f$ has the boundary normalization
$$
\lim_{t\to\infty} h^{-1}(h(z)+t)=1
$$
and $h(\mathbb{D})$ is a domain such that $\{w + t : t\ge 0\}\subseteq  h(\mathbb{D})$ for every $w\in h(\mathbb{D})$. Clearly, the image domain $h(\mathbb{D})$ is convex in the positive direction of the real axis. Denote by $\mathcal{CR}^+:=\mathcal{K}_{0}(k)$ the class of close-to-convex functions with argument $0$ and with respect to Koebe function $k(z)$. That is
$$
\mathcal{CR}^+=\left\{f\in\mathcal{A}: {\rm Re\,} (1-z)^2f'(z)>0, ~~ z\in\mathbb{D} \right\}.
$$
Then clearly functions in $\mathcal{CR}^+$ are convex in the positive direction of the real axis.
In the present article, we determine the upper bound of $|\gamma_3|$ for functions in $\mathcal{K}_0$ and $\mathcal{CR}^+$.

\section{Main Results}\label{Main Results}

Let $\mathcal{P}$ denote the class of analytic functions $P$ with positive real part on $\mathbb{D}$ which has the form
\begin{equation}\label{p6-025}
P(z)= 1+\sum_{n=1}^{\infty}c_n z^n.
\end{equation}
Functions in $\mathcal{P}$ are sometimes called Carath\'{e}odory function. To prove our main results, we need some preliminary lemmas. The first one is known as Carath\'{e}odory's lemma (see \cite[p. 41]{Duren-book} for example) and the second one is due to Libera and Z{\l}otkiewicz \cite{Libera-Zlotkiewicz-1982}.

\begin{lem}\cite[p. 41]{Duren-book}\label{p6-lemma001}
For a function $P\in\mathcal{P}$ of the form (\ref{p6-025}), the sharp inequality $|c_n|\le 2$ holds for each $n\ge 1$. Equality holds for the function $P(z)=(1+z)/(1-z)$.

\end{lem}

%


\begin{lem}\cite{Libera-Zlotkiewicz-1982}\label{p6-lemma010}
Let $P\in\mathcal{P}$ be of the form (\ref{p6-025}). Then there exist $x, t\in\mathbb{C}$ with $|x|\le 1$ and $|t|\le 1$ such that
$$
2c_2 = c_1^2 + x(4 - c_1^2)
$$
and
$$
4c_3= c_1^3+ 2(4-c_1^2)c_1x-c_1(4-c_1^2)x^2+2(4-c_1^2)(1-|x|^2)t.
$$

\end{lem}

In \cite{Thomas-2016},  Thomas claimed that his result (i.e. $|\gamma_3|\le 7/12$) is sharp for functions in the class $\mathcal{K}_0$
by ascertaining the equality holds for a function $f$ defined by $zf'(z)=g(z)P(z)$ where $g\in\mathcal{S}^*$ with $b_2(g)=b_3(g)=b_4(g)=2$ and $P\in\mathcal{P}$ with $c_1(P)=0$, $c_2(P)=c_3(P)=2$. But in view of  Lemma \ref{p6-lemma010}, it is easy to see that there does not exist a function $P\in\mathcal{P}$ with the property $c_1(P)=0$, $c_2(P)=c_3(P)=2$. Thus we can conclude that the result obtained by Thomas is not sharp. The main aim of the present  paper is to obtain a better upper bound for $|\gamma_3|$ for functions in the class $\mathcal{K}_0$ than that of obtained by Thomas \cite{Thomas-2016}. To prove our main results we also need the following Fekete-Szeg\"{o} inequality for functions in the class $\mathcal{S}^*$.

\begin{lem}\cite[Lemma 3]{Koepf-1987}\label{p6-lemma015}
Let $g\in\mathcal{S}^*$ be of the form $g(z)=z+\sum_{n=2}^{\infty}b_n z^n$. Then for any $\lambda\in\mathbb{C}$,
$$
|b_3-\lambda b_2^2|\le \max\{1,|3-4\lambda|\}.
$$
The inequality is sharp for $k(z)=z/(1-z)^2$ if $|3-4\lambda|\ge 1$ and for $(k(z^2))^{1/2}$ if $|3-4\lambda|<1$.

\end{lem}




For $f\in\mathcal{K}_0$ (close-to-convex functions with argument $0$), we obtained the following improved result for $|\gamma_3|$ (compare \cite{Thomas-2016}).

\begin{thm}\label{p6-theorem-001}
If $f\in\mathcal{K}_0$ then $|\gamma_3|\le \frac{1}{18} (3+4 \sqrt{2})=0.4809$.
\end{thm}

\begin{proof}
Let $f\in\mathcal{K}_0$ be of the form (\ref{p6-001}). Then there exists a starlike function $g(z)=z+\sum_{n=2}^{\infty}b_n z^n$ and a Carath\'{e}odory function $P\in\mathcal{P}$ of the form (\ref{p6-025}) such that
\begin{equation}\label{p6-029a}
zf'(z)=g(z)P(z).
\end{equation}
A comparison of the coefficients on the both sides of (\ref{p6-029a}) yields
\begin{align*}
a_2&=\frac{1}{2}(b_2+c_1)\\
a_3&=\frac{1}{3}(b_3+b_2c_1+c_2)\\
a_4&=\frac{1}{4}(b_4+b_3c_1+b_2c_2+c_3).
\end{align*}
By substituting the above $a_2, a_3$ and $a_4$  in (\ref{p6-020}) and then further simplification gives
\begin{align}\label{p6-030}
2\gamma_3
&= a_4-a_2a_3+\frac{1}{3}a_2^3\\
&=\frac{1}{24}\left((6b_4-4b_2b_3+b_2^3)+\frac{c_1}{2}\left(b_3-\frac{1}{2}b_2^2\right)+b_2(2c_2-c_1^2)+c_1^3-4c_1c_2+6c_3\right)\nonumber.
\end{align}
In view of Lemma \ref{p6-lemma010} and writing $c_2$ and $c_3$ in terms of $c_1$ we obtain
\begin{align}\label{p6-035}
48\gamma_3&= (6b_4-4b_2b_3+b_2^3)+ 2c_1\left(b_3-\frac{1}{2}b_2^2\right)+ b_2x(4-c_1^2)\\
&\quad +\frac{1}{2}c_1^3+c_1x(4-c_1^2)-\frac{3}{2}c_1x^2(4-c_1^2)+3(4-c_1^2)(1-|x|^2)t,\nonumber
\end{align}
where $|x|\le 1$ and $|t|\le 1$. Note that if $\gamma_3(g)$ denote the third logarithmic coefficient of $g\in\mathcal{S}^*$ then $|\gamma_3(g)|=\frac{1}{2}|b_4-b_2b_3+\frac{1}{3}b_2^3|\le \frac{1}{3}$. Since $g\in\mathcal{S}^*$, in view of Lemma \ref{p6-lemma015} we obtain
\begin{equation}\label{p6-040}
|6b_4-4b_2b_3+b_2^3|\le 6|b_4-b_2b_3+\frac{1}{3}b_2^3|+2|b_2||b_3-\frac{1}{2}b_2^2|\le 8.
\end{equation}
Since the class $\mathcal{K}_0$ is invariant under rotation, without loss of generality we can assume that $c_1=c$, where $0\le c\le 2$. Taking modulus on both the sides of (\ref{p6-035}) and then applying triangle inequality and further using the inequality (\ref{p6-040}) and Lemma \ref{p6-lemma015}, it follows that
$$
48|\gamma_3|\le 8+ 2c+ 2|x|(4-c^2)+\left|\frac{1}{2}c^3+cx(4-c^2)-\frac{3}{2}cx^2(4-c^2)\right|+3(4-c^2)(1-|x|^2),
$$
where we have also used the fact $|t|\le 1$. Let $x=re^{i\theta}$ where $0\le r\le 1$ and $0\le\theta\le 2\pi$. For simplicity, by writing $\cos\theta=p$ we obtain
\begin{equation}\label{p6-042}
48|\gamma_3|\le \psi(c,r)+\left|\phi(c,r,p)\right|=:F(c,r,p)
\end{equation}
where $\psi(c,r)=8+ 2c+ 2r(4-c^2)+3(4-c^2)(1-r^2)$ and
\begin{align*}
\phi(c,r,p)
&=\left(\frac{1}{4}c^6+c^2r^2(4-c^2)^2+\frac{9}{4}c^2r^4(4-c^2)^2+c^4(4-c^2)rp\right.\\
&\qquad\quad \left.-\frac{3}{2}c^4r^2(4-c^2)(2p^2-1)-3c^2(4-c^2)r^3p \right)^{1/2}.
\end{align*}
Thus we need to find the maximum value of $F(c,r,p)$ over the rectangular cube $R:=[0,2]\times[0,1]\times[-1,1]$.

By elementary calculus one can verify the followings:
\begin{align*}
&\max_{0\le r\le 1} \psi(0,r)=\psi\left(0,\frac{1}{3}\right)=\frac{64}{3},\quad \max_{0\le r\le 1} \psi(2,r)=12,\\[2mm]
&\max_{0\le c\le 2} \psi(c,0)=\psi\left(\frac{1}{3},0\right)=\frac{61}{3},\quad \max_{0\le c\le 2} \psi(c,1)=\psi(0,1)=16 \quad\mbox { and }\\[2mm]
&\max_{(c,r)\in[0,2]\times[0,1]} \psi(c,r)=\psi\left(\frac{3}{10},\frac{1}{3}\right)=\frac{649}{30}=21.6333.
\end{align*}
We first find the maximum value of $F(c,r,p)$ on the boundary of $R$, i.e on the six faces of the rectangular cube $R$.

On the face $c=0$, we have $F(0,r,p)=\psi(0,r)$, where $(r,p)\in R_1:=[0,1]\times[-1,1]$. Thus
$$
\max_{(r,p)\in R_1} F(0,r,p)= \max_{0\le r\le 1} \psi(0,r)=\psi\left(0,\frac{1}{3}\right)=\frac{64}{3}=21.33.
$$

On the face $c=2$, we have $F(2,r,p)= 16$, where $(r,p)\in R_1$.

On the face $r=0$, we have $F(c,0,p)=8+ 2c+3(4-c^2)+\frac{1}{2}c^3$, where $(c,p)\in R_2:=[0,2]\times[-1,1]$. By using elementary calculus it is easy to see that
$$
\max_{(c,p)\in R_2} F(c,0,p)= F\left(\frac{2}{3} (3-\sqrt{6}),0,p\right)=\frac{16}{9} \left(9+\sqrt{6}\right)=20.3546.
$$

On the face $r=1$, we have $F(c,1,p)=\psi(c,1)+ |\phi(c,1,p)|$, where $(c,p)\in R_2$. We first prove that $\phi(c,1,p)\ne 0$ in the interior of $R_2$.
On the contrary, if $\phi(c,1,p)=0$ in the interior of $R_2$ then
$$
|\phi(c,1,p)|^2=\left|\frac{1}{2}c^3+ce^{i\theta}(4-c^2)-\frac{3}{2}ce^{2i\theta}(4-c^2)\right|^2=0
$$
and hence
\begin{equation}\label{p6-042a}
\frac{1}{2}c^3+cp(4-c^2)-\frac{3}{2}c(4-c^2)(2p^2-1)=0 ~\mbox{ and } c(4-c^2)\sin\theta-\frac{3}{2}c(4-c^2)\sin2\theta=0.
\end{equation}
Further, (\ref{p6-042a}) reduces to
$$
\frac{1}{2}c^2+p(4-c^2)-\frac{3}{2}(4-c^2)(2p^2-1)=0 \quad\mbox{and}\quad 1-3p=0,
$$
which is equivalent to $p=1/3$ and $c^2=6$. This contradicts the range of $c\in (0,2)$.  Thus $\phi(c,1,p)\ne 0$ in the interior of $R_2$.

Next, we prove that $F(c,1,p)$ has no maximum at any interior point of $R_2$. Suppose that $F(c,1,p)$  has the maximum at an interior point of $R_2$. Then at such point $\frac{\partial F(c,1,p)}{\partial c}=0$  and $\frac{\partial F(c,1,p)}{\partial p}=0$. From $\frac{\partial F(c,1,p)}{\partial p}=0$, (for points in the interior of $R_2$), a straight forward calculation gives
\begin{equation}\label{p6-050}
p=\frac{2 \left(c^2-3\right)}{3 c^2}.
\end{equation}
Substituting the value of $p$ as given in (\ref{p6-050}) in the relation $\frac{\partial F(c,1,p)}{\partial c}=0$ and further simplification gives
\begin{equation}\label{p6-052}
3c^3-2c+(2c-1) \sqrt{6(c^2+2)}=0.
\end{equation}
It is easy to show that the function $\rho(c)=3c^3-2c+(2c-1) \sqrt{6(c^2+2)}$ is strictly increasing in $(0,2)$. Since $\rho(0)<0$ and $\rho(2)>0$, the equation (\ref{p6-052}) has exactly one solution in $(0,2)$. By solving the equation (\ref{p6-052}) numerically, we obtain the approximate root in $(0,2)$ as $0.5772$. But the corresponding value of $p$ obtained by (\ref{p6-050}) is $-5.3365$ which does not belong to $(-1,1)$. Thus $F(c,1,p)$ has no maximum at any interior point of $R_2$.

Thus we find the maximum value of $F(c,1,p)$ on the boundary of $R_2$. Clearly, $F(0,1,p)=F(2,1,p)=16$,
$$
F(c,1,-1)=
\begin{cases}
8+ 2c+ 2(4-c^2)+c(10-3c^2)& \mbox{for}\quad 0\le c\le \sqrt{\frac{10}{3}}\\[2mm]
8+ 2c+ 2(4-c^2)-c(10-3c^2)& \mbox{for}\quad \sqrt{\frac{10}{3}}< c\le 2
\end{cases}
$$
and
$$
F(c,1,1)=
\begin{cases}
8+ 2c+ 2(4-c^2)+c(2-c^2)& \mbox{for}\quad 0\le c\le \sqrt{2}\\[2mm]
8+ 2c+ 2(4-c^2)-c(2-c^2)& \mbox{for}\quad \sqrt{2}< c\le 2.
\end{cases}
$$
By using elementary calculus we find that
$$
\max_{0\le c\le 2}F(c,1,-1)= F\left(\frac{2}{9} (2 \sqrt{7}-1),1,-1\right)=\frac{8}{243} \left(403+112 \sqrt{7}\right)=23.023\quad\mbox{ and }
$$
$$
\max_{0\le c\le 2}F(c,1,1)= F\left(\frac{2}{3},1,1\right)=\frac{427}{27} =17.48.
$$
Hence,
$$
\max_{(c,p)\in R_2} F(c,1,p)= F\left(\frac{2}{9} (2 \sqrt{7}-1),1,-1\right)=\frac{8}{243} \left(403+112 \sqrt{7}\right)=23.023.
$$

On the face $p=-1$,
$$
F(c,r,-1)=
\begin{cases}
\psi(c,r)+\eta_1(c,r) & \mbox{ for }\quad \eta_1(c,r)\ge 0\\[2mm]
\psi(c,r)-\eta_1(c,r)& \mbox{ for }\quad \eta_1(c,r)< 0,
\end{cases}
$$
where $\eta_1(c,r)=c^3 (3r^2+2r+1)-4cr(3r+2)$ and $(c,r)\in R_3:=[0,2]\times[0,1]$. Differentiating partially $F(c,r,-1)$ with respect to $c$ and $r$ and a routine calculation shows that
$$
\max_{(c,r)\in  {\rm int \,}R_3\setminus S_1} F(c,r,-1)= F\left(2(\sqrt{2}-1),\frac{1}{3}(1+\sqrt{2}),-1\right)=\frac{8}{3} (3+4 \sqrt{2})=23.0849,
$$
where $S_1=\{(c,r)\in R_3: \eta_1(c,r)=0\}$. Now we find the maximum value of $F(c,r,-1)$ on the boundary of $R_3$ and on the set $S_1$. Note that
$$
\max_{(c,r)\in S_1} F(c,r,-1)\le \max_{(c,r)\in R_3} \psi(c,r)=\frac{649}{30}=21.6333.
$$
On the other hand by using elementary calculus, as before, we find that
$$
\max_{(c,r)\in \partial R_3} F(c,r,-1)=F\left(\frac{2}{9} (2 \sqrt{7}-1),1,-1\right)=\frac{8}{243} \left(403+112 \sqrt{7}\right)=23.023,
$$
where $\partial R_3$ denotes the boundary of $R_3$. Hence, by combining the above cases we obtain
$$
\max_{(c,r)\in  R_3} F(c,r,-1)=F\left(2(\sqrt{2}-1),\frac{1}{3}(1+\sqrt{2}),-1\right)=\frac{8}{3} (3+4 \sqrt{2})=23.0849.
$$

On the face $p=1$,
$$
F(c,r,1)=
\begin{cases}
\psi(c,r)+\eta_2(c,r) & \mbox{ for }\quad \eta_2(c,r)\ge 0\\[2mm]
\psi(c,r)-\eta_2(c,r)& \mbox{ for }\quad \eta_2(c,r)< 0,
\end{cases}
$$
where $\eta_2(c,r)=c^3 (3r^2-2r+1)+4cr(3r-2)$ and $(c,r)\in R_3$. Differentiating partially $F(c,r,1)$ with respect to $c$ and $r$ and a routine calculation shows that
$$
\max_{(c,r)\in  {\rm int \,}R_3\setminus S_2} F(c,r,1)=F\left(\frac{1}{3} (10-2 \sqrt{19}),\frac{1}{3},1\right)=\frac{16}{81} \left(28+19 \sqrt{19}\right)=21.89,
$$
where $S_2=\{(c,r)\in R_3: \eta_2(c,r)=0\}$. Now, we find the maximum value of $F(c,r,1)$ on the boundary of $R_3$ and on the set $S_2$. By noting that
$$
\max_{(c,r)\in S_2} F(c,r,1)\le \max_{(c,r)\in R_3} \psi(c,r)=\frac{649}{30}=21.6333
$$
and proceeding similarly as in the previous case, we find that
$$
\max_{(c,r)\in  R_3} F(c,r,1)=F\left(\frac{1}{3} (10-2 \sqrt{19}),\frac{1}{3},1\right)=\frac{16}{81} \left(28+19 \sqrt{19}\right)=21.89.
$$

Let $S'=\{(c,r,p)\in R: \phi(c,r,p)=0\}$. Then
$$
\max_{(c,r,p)\in S'} F(c,r,p)\le \max_{(c,r)\in R_3} \psi(c,r)=\psi\left(\frac{3}{10},\frac{1}{3}\right)=\frac{649}{30}=21.6333.
$$
We prove that $F(c,r,p)$ has no maximum value at any interior point of $R\setminus S'$. Suppose that $F(c,r,p)$  has a maximum value at an interior point of $R\setminus S'$. Then at such point $\frac{\partial F}{\partial c}=0$, $\frac{\partial F}{\partial r}=0$ and $\frac{\partial F}{\partial p}=0$. Note that $\frac{\partial F}{\partial c}$, $\frac{\partial F}{\partial r}$ and $\frac{\partial F}{\partial p}$ may not exist at points in $S'$. In view of  $\frac{\partial F}{\partial p}=0$ (for points in the interior of $R\setminus S'$), a straight forward but laborious calculation gives
\begin{equation}\label{p6-060}
p= \frac{3 c^2 r^2+c^2-12 r^2}{6 c^2 r}.
\end{equation}
Substituting the value of $p$ as given in  (\ref{p6-060}) in the relations $\frac{\partial F}{\partial c}=0$ and $\frac{\partial F}{\partial r}=0$ and simplifying (again, a long and laborious calculation), we obtain
\begin{equation}\label{p6-065}
\frac{3 \sqrt{6} c^3 (1-3 r^2)+12 (c(3r^2-2r-3)+1)\sqrt{c^2+2} )+4\sqrt{6}c}{6\sqrt{c^2+2}}=0
\end{equation}
and
\begin{equation}\label{p6-070}
(4-c^2) \left(( \sqrt{6(c^2+2)}-6) r+2\right)=0.
\end{equation}
Since $0<c<2$, solving the equation (\ref{p6-070}) for $r$, we obtain
\begin{equation}\label{p6-075}
r=\frac{2}{6-\sqrt{6(c^2+2)}}.
\end{equation}
Substituting the value of $r$ in (\ref{p6-065}) and then further  simplification gives
$$
3 c^3+6 c -(6 c-2) \sqrt{6 \left(c^2+2\right)}=0.
$$
Taking the last term on the right hand side and squaring on both sides yields
\begin{equation}\label{p6-078}
3 \left(c^2+2\right) \left(3 c^4-66 c^2+48 c-8\right)=0.
\end{equation}
Clearly $c^2+2\ne 0$ in $0<c<2$. On the other hand the polynomial $q(c)=3 c^4-66 c^2+48 c-8$ has exactly two roots in $(0,2)$, one lies in $(0,1/3)$ and another lies in $(1/3,1/2)$. This can be seen using the well-known Strum theorem for isolating real roots  and hence for the sake of brevity we omit the details. By solving the equation $q(c)=0$ numerically, we obtain two approximate roots $0.2577$ and $0.4795$ in $(0,2)$. But the corresponding value of $p$ obtained from (\ref{p6-075}) and (\ref{p6-060}) are $-23.6862$ and $-6.80595$ which do not belong to $(-1,1)$. This proves that $F(c,r,p)$ has no maximum in the interior of $R\setminus S'$

Thus combining all the above cases we find that
$$
\max_{(c,r,p)\in  R} F(c,r,p)=F\left(2(\sqrt{2}-1),\frac{1}{3}(1+\sqrt{2}),-1\right)=\frac{8}{3} (3+4 \sqrt{2})=23.0849,
$$
and hence from (\ref{p6-042}) we obtain
$$
|\gamma_3|\le \frac{1}{18} (3+4 \sqrt{2})=0.4809.
$$

\end{proof}

We obtained the following sharp upper bound for $|\gamma_3|$ for functions in the class $\mathcal{CR}^+$.
\begin{thm}\label{p6-theorem-002}
Let $f\in\mathcal{CR}^+$ be of the form (\ref{p6-001}) with $1\le a_2 \le 2$. Then
\begin{equation}\label{p6-088}
|\gamma_3|\le \frac{1}{243} (28+19 \sqrt{19})=0.4560.
\end{equation}
The inequality is sharp.
\end{thm}

\begin{proof}
If $f\in\mathcal{CR}^+$ then there exists a Carath\'{e}odory function $P\in\mathcal{P}$ of the form (\ref{p6-025}) such that $zf'(z)=g(z)P(z)$, where $g(z):=k(z)=z/(1-z)^2$. Following the same method as used in Theorem \ref{p6-theorem-001} and noting that  $g(z):=k(z)=z+2z^2+3z^3+4z^4+\cdots$, a simple computation in (\ref{p6-035}) shows that
\begin{equation}\label{p6-090}
48\gamma_3= 8+ 2c_1+\frac{1}{2}c_1^3+ (4-c_1^2)(2x +c_1x-\frac{3}{2}c_1x^2)+3(4-c_1^2)(1-|x|^2)t,
\end{equation}
where $|x|\le 1$ and $|t|\le 1$. Since $1\le a_2 \le 2$ and $2a_2=2+c_1$, then $0\le c_1\le 2$. Taking modulus on the both sides of  (\ref{p6-090}) and then applying triangle inequality and writing $c=c_1$, it follows that
$$
48|\gamma_3|\le \left|8+ 2c_1+\frac{1}{2}c_1^3+ (4-c_1^2)(2x +c_1x-\frac{3}{2}c_1x^2)\right|+3(4-c^2)(1-|x|^2),
$$
where we have also used the fact $|t|\le 1$. Let $x=re^{i\theta}$ where $0\le r\le 1$ and $0\le\theta\le 2\pi$. For simplicity, by writing $\cos\theta=p$ we obtain
\begin{equation}\label{p6-095}
48|\gamma_3|\le \psi(c,r)+\left|\phi(c,r,p)\right|=:F(c,r,p)
\end{equation}
where $\psi(c,r)=3(4-c^2)(1-r^2)$ and
\begin{align*}
\phi(c,r,p)
&=\left( (8+2c+\frac{1}{2}c^3)^2 +r^2(4-c^2)^2(4+c^2+\frac{9}{4}c^2r^2+4c-6crp-3c^2rp)\right.\\
&\qquad\quad \left. +2(4-c^2)(8+2c+\frac{1}{2}c^3)(2rp+crp-\frac{3}{2}cr^2(2p^2-1)) \right)^{1/2}.
\end{align*}
Thus we need to find the maximum value of $F(c,r,p)$ over the rectangular cube $R=[0,2]\times[0,1]\times[-1,1]$.

We first find the maximum value of $F(c,r,p)$ on the boundary of $R$, i.e on the six faces of the rectangular cube $R$. As before, let $R_1=[0,1]\times[-1,1], R_2=[0,2]\times[-1,1]$ and $R_3=[0,2]\times[0,1]$. By elementary calculus it is not very difficult to prove that
\begin{align*}
\max_{(r,p)\in R_1} F(0,r,p)&= F(0,\frac{1}{3},1)=\frac{64}{3}=21.33,\\
\max_{(r,p)\in R_1} F(2,r,p)&= F(2,r,p)= 16,\\
\max_{(c,p)\in R_2} F(c,0,p)&= F\left(\frac{2}{3} (3-\sqrt{6}),0,p\right)=\frac{16}{9} \left(9+\sqrt{6}\right)=20.3546.
\end{align*}

On the face $r=1$, we have $F(c,1,p)=|\phi(c,1,p)|$ where $(c,p)\in R_2$. As in the proof of Theorem \ref{p6-theorem-001}, one can verify that $\phi(c,1,p)\ne 0$ in the interior of $R_2$ (otherwise, one can simply proceed to find maximum value $F(c,1,p)$ at an interior point of $R_2\setminus T$, where $T=\{(c,p)\in R_2: \phi_1(c,1,p)=0\}$, as $F(c,1,p)=0$ in $T$). Suppose that $F(c,1,p)$  has the maximum value at an interior point of $R_2$. Then at such point $\frac{\partial F}{\partial c}=0$ and $\frac{\partial F}{\partial p}=0$. From $\frac{\partial F}{\partial p}=0$ (for points in the interior of $R_2$), it follows that
\begin{equation}\label{p6-100}
p=\frac{2 \left(c^3-2 c+4\right)}{3 c \left(c^2-2 c+8\right)}.
\end{equation}
By substituting the above  value of $p$ given in  (\ref{p6-100}) in the relation $\frac{\partial F}{\partial c}=0$ and further computation (a long and laborious calculation) gives
\begin{equation*}\label{p6-105}
3 c^8-17 c^7+76 c^6-136 c^5+120 c^4+640 c^3-832 c^2-192 c+128=0.
\end{equation*}
This equation has exactly two real roots in $(0,2)$, one lies in $(0,1)$ and another lies in $(1,2)$. This can be seen using the well-known Strum theorem for isolating real roots  therefore  for the sake of brevity we omit the details. Solving this equation numerically we obtain two approximate roots $0.3261$ and $1.2994$ in $(0,2)$  and the corresponding values of $p$ are $0.9274$ and $0.2602$ respectively. Thus the extremum points of $F(c,1,p)$ in the interior of $R_2$ lie in a small neighborhood of the points $A_1=(0.3261,1,0.9274)$ and $A_2=(1.2994,1,0.2602)$ (on the plane $r=1$). Now $F(A_1)=15.8329$ and $F(A_2)=18.6303$. Since the function $F(c,1,p)$ is uniformly continuous on $R_2$, the value of $F(c,1,p)$ would not vary too much in the neighborhood of the points $A_1$ and $A_2$. Again, proceeding similarly as in the proof of Theorem \ref{p6-theorem-001}, we find that
$$
\max_{(c,p)\in \partial R_2} F(c,1,p)= F(2,1,p)=16
$$
and hence
$$
\max_{(c,p)\in R_2} F(c,1,p) \thickapprox 18.6306< \frac{64}{3}.
$$

On the face $p=-1$,
$$
F(c,r,-1)=
\begin{cases}
\psi(c,r)+\eta_1(c,r) & \mbox{ for }\quad \eta_1(c,r)\ge 0\\[2mm]
\psi(c,r)-\eta_1(c,r)& \mbox{ for }\quad \eta_1(c,r)\le 0,
\end{cases}
$$
where $\eta_1(c,r)=c^3-3cr^2 (4-c^2)+2 (c-2) (c+2)^2 r+4 c+16$ and $(c,r)\in R_3$. Again, proceeding similarly as in the proof of Theorem \ref{p6-theorem-001}, we can show that $F(c,r,-1)$ has no maximum in the interior of $R_3\setminus S_1$, where $S_1=\{(c,r)\in R_3: \eta_1(c,r)=0\}$. Computing the maximum value on the boundary of $R_3$ and on the set $S_1$ we conclude that
$$
\max_{(c,r)\in R_3} F(c,r,-1)= F(0,0,-1)=20.
$$

On the face $p=1$, we have $F(c,r,1)=\psi(c,r)+\eta_2(c,r)$, where
\begin{align*}
\eta_2(c,r)&=(c+2) (8-2c+c^2+8r-2c^2r-6cr^2+3c^2r^2)\\
& \ge (c+2)\left(3+(1-c)^2+r(8-2c^2)+r^2(3c^2-6c+4)\right)\\
&\ge 0
\end{align*}
for $(c,r)\in R_3$. Differentiating partially $F(c,r,1)$ with respect to $c$ and $r$ and a routine calculation shows that
$$
\max_{(c,r)\in  {\rm int \,}R_3} F(c,r,1)=F\left(\frac{1}{3} (10-2 \sqrt{19}),\frac{1}{3},1\right)=\frac{16}{81} \left(28+19 \sqrt{19}\right)=21.8902,
$$
and on the boundary of $R_3$ we have
$$
\max_{(c,r)\in  \partial R_3} F(c,r,1)=F(0,\frac{1}{3},1)=\frac{64}{3}=21.33.
$$
Thus,
$$
\max_{(c,r)\in R_3} F(c,r,1)=F\left(\frac{1}{3} (10-2 \sqrt{19}),\frac{1}{3},1\right)=\frac{16}{81} \left(28+19 \sqrt{19}\right)=21.8902.
$$

Let $S'=\{(c,r,p)\in R: \phi(c,r,p)=0\}$. Then
$$
\max_{(c,r,p)\in S'} F(c,r,p)\le \max_{(c,r)\in R_3} \psi(c,r)=12.
$$
We now prove that $F(c,r,p)$ has no maximum at an interior point of $R\setminus S'$. Suppose that $F(c,r,p)$  has a maximum at an interior point of $R\setminus S'$. Then at such point $\frac{\partial F}{\partial c}=0$, $\frac{\partial F}{\partial r}=0$ and $\frac{\partial F}{\partial p}=0$. Note that $\frac{\partial F}{\partial c}$, $\frac{\partial F}{\partial r}$ and $\frac{\partial F}{\partial p}$ may not exist at points in $S'$. In view of  $\frac{\partial F}{\partial p}=0$ (for points in the interior of $R\setminus S'$), a straight forward but laborious calculation gives
\begin{equation}\label{p6-125}
p= \frac{3 c^3 r^2+c^3-12 c r^2+4 c+16}{6cr (c^2-2 c+8)}.
\end{equation}
Substituting the value of $p$ given in  (\ref{p6-125}) in the relation $\frac{\partial F}{\partial r}=0$ and then further simplifying (again, a long and laborious calculation), we obtain
\begin{equation}\label{p6-130}
r(4-c^2) \left( c \sqrt{\frac{6(c^3-4 c^2+14 c+4)}{c(c^2-2 c+8)}}-6\right)=0.
\end{equation}
Since $0<c<2$ and $0<r<1$, we can divide by $r(4-c^2)$ on both the sides of (\ref{p6-130}). Further, a simple computation shows that
\begin{equation*}\label{p6-135}
\frac{6 (4-c^2) (c^2-4 c+12)}{c^2-2 c+8}=0.
\end{equation*}
But this equation has no real roots in $(0,2)$. Therefore, $F(c,r,p)$ has no maximum at an interior point of $R\setminus S'$.

Thus combining all the cases we find that
$$
\max_{(c,r,p)\in R} F(c,r,p)=F\left(\frac{1}{3} (10-2 \sqrt{19}),\frac{1}{3},1\right)=\frac{16}{81} \left(28+19 \sqrt{19}\right)=21.8902,
$$
and hence, from (\ref{p6-095}) we obtain
$$
|\gamma_3|\le \frac{1}{243} (28+19 \sqrt{19})=0.4560.
$$

We now show that the inequality (\ref{p6-088}) is sharp. It is pertinent to note that  equality holds in (\ref{p6-088}) if we choose $c_1=c=\frac{1}{3} (10-2 \sqrt{19})$, $x=1$ and $t=1$ in (\ref{p6-090}). For such values of $c_1, x$ and $t$,  Lemma \ref{p6-lemma010} elicit  $c_2=\frac{2}{27} (97-20 \sqrt{19})$ and $c_3=\frac{1}{243} (2050-362 \sqrt{19})$. A function $P\in\mathcal{P}$ having the first three coefficients $c_1, c_2$ and $c_3$ as above is given by
\begin{align}\label{p6-140}
P(z) &=(1-2\lambda) \frac{1+z}{1-z}+\lambda \frac{1+uz}{1-uz}+\lambda \frac{1+\overline{u}z}{1-\overline{u}z}\\
& =1+\frac{1}{3} (10-2 \sqrt{19})z +\frac{2}{27} (97-20 \sqrt{19})z^2 +\frac{1}{243} (2050-362 \sqrt{19})z^3+\cdots,\nonumber
\end{align}
where $\lambda=\frac{1}{9} (-1-\sqrt{19})$ and $u=\alpha+i\sqrt{1-\alpha^2}$ with $\alpha=\frac{1}{18} (-13+4 \sqrt{19})$. Hence the inequality (\ref{p6-088}) is sharp for a function $f$ defined by $(1-z)^2f'(z)=P(z)$, where $P(z)$ is given by (\ref{p6-140}). This completes the proof.

\end{proof}

\begin{rem}
In \cite{Thomas-2016}, Thomas proved that $|\gamma_3|\le \frac{7}{12}=0.5833$ for functions in the class $\mathcal{K}_0$ with an additional condition that the second coefficient $b_2$ of the corresponding starlike function $g$ is real. However, in Theorem \ref{p6-theorem-001} we obtained a much improved bound $|\gamma_3|\le \frac{1}{18} (3+4 \sqrt{2})=0.4809$ for functions in the whole class $\mathcal{K}_0$ without assuming any additional condition on functions in the class $\mathcal{K}_0$. While for functions in the class $\mathcal{CR}^+$ (with $1\le a_2\le 2$) we obtained the sharp bound $|\gamma_3|\le \frac{1}{243} (28+19 \sqrt{19})=0.4560$. We conjecture that for the whole class $\mathcal{K}_0$ the sharp upper bound for $|\gamma_3|$ is $|\gamma_3|\le \frac{1}{243} (28+19 \sqrt{19})=0.4560$.

\end{rem}

\vspace{4mm}
\noindent\textbf{Acknowledgement:}
The first author thank University Grants Commission for the financial support through UGC-SRF Fellowship. The second author thank SERB (DST) for financial support.

\end{document}